\documentclass[letterpaper, 10 pt, conference]{ieeeconf}  

\usepackage{amsmath, amssymb, bbm, xspace}

\def\det{\mathop{\hbox{\rm det}}}
\def\diag{\mathop{\hbox{\rm diag}}}

\def\Null{\mathop{\hbox{\rm null}}}

\def\spose#1{\hbox to 0pt{#1\hss}}

\def\text #1{\hbox{\quad#1\quad}}

\def\nthinsp{\mskip -2   mu}

\def\F{_{\scriptscriptstyle F}}

\def\k{_k}

\def\N{\mathbb{N}}

\def\superstar{^{\raise 0.5pt\hbox{$\nthinsp *$}}}
\def\SUPERSTAR{^{\raise 0.5pt\hbox{$*$}}}

\def\lamstarT {\lambda^{\raise 0.5pt\hbox{$\nthinsp *$}T}}

\def\mubar{\skew3\bar \mu}

\def\sigmabar{\bar\sigma}

\def\hbar{\skew{4.2}\bar h}

		\def\bkE{{\rm I\kern-.17em E}}
		\def\bk1{{\rm 1\kern-.17em l}}
		\def\bkD{{\rm I\kern-.17em D}}
		\def\bkR{{\rm I\kern-.17em R}}
		\def\bkP{{\rm I\kern-.17em P}}
		\def\bkY{{\bf \kern-.17em Y}}
		\def\bkZ{{\bf \kern-.17em Z}}

		\def\beq{\begin{eqnarray}}
		\def\bc{\begin{center}}
		\def\be{\begin{enumerate}}
		\def\bi{\begin{itemize}}
		\def\bS{\begin{slide}}
		\def\ec{\end{center}}
		\def\ee{\end{enumerate}}
		\def\ei{\end{itemize}}
		\def\eS{\end{slide}}
		\def\eeq{\end{eqnarray}}

	\def\cp2problem#1#2#3#4{\fbox
		 {\begin{tabular*}{0.9\textwidth}
			{@{}l@{\extracolsep{\fill}}l@{\extracolsep{6pt}}l@{\extracolsep{\fill}}c@{}}
				#1 & & $#4 $ 
			\end{tabular*}}}

		\renewcommand{\emph}[1]{\textbf{#1}}

		\def\bkE{{\rm I\kern-.17em E}}
		\def\bk1{{\rm 1\kern-.17em l}}
		\def\bkD{{\rm I\kern-.17em D}}
		\def\bkR{{\rm I\kern-.17em R}}
		\def\bkP{{\rm I\kern-.17em P}}
		
		\def\bkZ{{\bf{Z}}}

\newcommand {\beeq}[1]{\begin{equation}\label{#1}}
\newcommand {\eeeq}{\end{equation}}
\newcommand {\bea}{\begin{eqnarray}}
\newcommand {\eea}{\end{eqnarray}}

\def\texitem#1{\par\smallskip\noindent\hangindent 25pt
               \hbox to 25pt {\hss #1 ~}\ignorespaces}

\newtheorem{definition}{Definition}{\it}{}
{\it}{}
{\it}{}
{\it}{}
\newtheorem{lemma}{Lemma}{\it}{}
\newtheorem{theorem}{Theorem}{\it}{}
\newtheorem{remark}{Remark}{\it}{}
\newtheorem{assumption}{Assumption}{\it}{}
\newtheorem{standing}{Standing Assumption}

\def\diag{{\rm diag}}

\def\R{\mathbb{R}}

\def\F{\mathcal{F}}

\def\argmin{\mathop{\rm argmin}}

\def\x{\bs{x}}
\def\xhb{{\hat{\bs{x}}}}
\def\y{\bs{y}}

\def\F{{\bs{F}}}

\def\ber{\begin{eqnarray}}
\def\eer{\end{eqnarray}}
\def\bers{\begin{eqnarray*}}
	\def\eers{\end{eqnarray*}}
\def\be{\begin{equation}}
\def\ee{\end{equation}}
\def\1{{\bf 1}}

\newcommand{\bs}{\boldsymbol}
\newcommand{\mc}{\mathcal}

\newcommand{\proj}{\mathrm{proj}}
\newcommand{\col}{\mathrm{col}}

\renewcommand{\emph}{\textit}

\newcommand{\0}{\bs 0}

\def\k{{k \in \N}}

\newcommand{\alphamax}{\alpha_{\textnormal{max}}}
\newcommand{\gammamax}{\gamma_{\textnormal{max}}}
\newcommand{\taumax}{\tau_{\textnormal{max}}}

\newcommand{\Fa}{\F_{\! \textnormal{a}}}
\def\W{\bs{W}}

\usepackage{amsmath} 
\usepackage{amsfonts,mathrsfs}
\usepackage{amssymb}
\usepackage{mathtools}
\usepackage{color}
\usepackage{graphicx}
\usepackage{epsfig}
\usepackage{algorithm}
\usepackage{algorithmicx}
\usepackage[acronym]{glossaries}
\usepackage{subfigure}
\usepackage{cancel}
\usepackage{empheq}
\usepackage{placeins}

\let\labelindent\relax
\usepackage{enumitem}

\makeglossaries
\newacronym{GNEP}{GNEP}{generalized Nash equilibrium problem}
\newacronym{NE}{NE}{Nash equilibrium}
\newacronym{NEP}{NEP}{Nash equilibrium problem}
\newacronym{GNE}{GNE}{generalized Nash equilibrium}
\newacronym{v-GNE}{v-GNE}{variational \gls{GNE}}
\newacronym{ISS}{ISS}{input-to-state-stable}
\newacronym{PPPA}{PPPA}{preconditioned proximal-point algorithm}
\newacronym{PPA}{PPA}{proximal-point algorithm}
\newacronym{VI}{VI}{variational inequality}
\newacronym{GAE}{GAE}{generalized aggregative equilibrium}
\newacronym{v-GAE}{v-GAE}{variational \gls{GAE}}
\newacronym{KKT}{KKT}{Karush-Kuhn-Tucker}
\newacronym{FQNE}{FQNE}{firmly quasinonexpansive}
\newacronym{FNE}{FNE}{firmly nonexpansive}
\DeclareSymbolFont{myletters}{OML}{ztmcm}{m}{it}
\DeclareMathSymbol{\uplambda}{\mathord}{myletters}{"15}

\def\QEDhereeqn{\eqno\let\eqno\relax\let\leqno\relax\let\veqno\relax\hbox{\QED}}
\def\QEDopenhereeqn{\eqno\let\eqno\relax\let\leqno\relax\let\veqno\relax\hbox{\QEDopen}}

\def\QEDhereeqn{\eqno\let\eqno\relax\let\leqno\relax\let\veqno\relax\hbox{\QED}}
\def\QEDopenhereeqn{\eqno\let\eqno\relax\let\leqno\relax\let\veqno\relax\hbox{\QEDopen}}
\IEEEoverridecommandlockouts 
\title{\LARGE \bf
 Fully distributed Nash equilibrium seeking over time-varying communication networks with linear convergence rate
}

\author{Mattia Bianchi and Sergio Grammatico
\thanks{The authors are with the Delft Center for Systems and Control (DCSC), TU Delft, The Netherlands.
	E-mail addresses: \texttt{\{m.bianchi, s.grammatico\}@tudelft.nl}. This work was partially supported by NWO under research project OMEGA (grant n. 613.001.702) and by the ERC under research project COSMOS (802348).} 
}

\begin{document}

\maketitle
\thispagestyle{empty}
\pagestyle{empty}

\begin{abstract}
	We design a distributed algorithm for learning Nash equilibria
	over time-varying communication networks
	in a partial-decision information scenario, where each agent can access its own cost function and local feasible set, but can only observe the actions of some neighbors. 
	Our algorithm is based on projected pseudo-gradient dynamics, augmented with consensual terms. Under strong monotonicity and Lipschitz continuity of the game mapping, we provide a simple proof of linear convergence, based on a contractivity property of the iterates. Compared to similar solutions proposed in literature, we also  allow for time-varying communication and derive tighter bounds on the step sizes that ensure convergence. In fact, in our numerical simulations, our algorithm outperforms the existing gradient-based methods, when the step sizes are set to their theoretical upper bounds. Finally, to relax the assumptions on the network structure, we propose a different pseudo-gradient algorithm, which is guaranteed to converge on time-varying balanced directed graphs. 
\end{abstract}

\section{Introduction}
\gls{NE} problems arise in several network systems, where multiple selfish decision-makers, or agents, aim at optimizing their individual, yet inter-dependent, objective functions. Engineering applications include communication networks \cite{Palomar_Eldar_Facchinei_Pang_2009}, demand-side management in the smart grid \cite{Saad2012}, charging of electric vehicles \cite{Grammatico2017} and demand response in competitive markets \cite{Li_Chen_Dahleh_2015}.
From a game-theoretic perspective, the challenge is to assign the agents behavioral rules that eventually ensure the attainment of a \gls{NE}, a joint action from which no agent has an incentive to unilaterally deviate. 
\newline\indent
\emph{Literature review:} Typically, \gls{NE} seeking algorithms are designed under the assumption that each agent can access the decisions of all the competitors  \cite{Yu_VanderSchaar_Sayed_2017}, \cite{BelgioiosoGrammatico_ECC_2018}, \cite{Shamma_Arslan_2005}.  
This \emph{full-decision information} hypothesis requires the presence of a coordinator, that broadcast the data to the network, and it is impractical for some applications  \cite{Ghaderi_2014}, \cite{Bimpikis2014}.
One example is the Nash-Cournot  competition model described in \cite{Koshal_Nedic_Shanbag_2016}, where the  profit of each of a group of firms depends not only on its own production, but also on the whole amount of sales, a quantity not directly accessible by any of the firms.
Therefore, in recent years, there has been an increased attention  for fully distributed algorithms that allow to compute \glspl{NE} relying on local information  only.
In this paper, we consider the so-called \emph{partial-decision information} scenario, where the agents
engage in nonstrategic information exchange 
with some neighbors on a network;
based on the data received, they can estimate and
eventually reconstruct the actions of all the competitors.
This setup has only been introduced very recently. In particular, most of the results available resort to (projected) gradient and consensus dynamics, both in continuous time \cite{YeHu2017}, \cite{GadjovPavel2018}, and discrete time. 
For the discrete time case,  fixed-step algorithms were proposed in \cite{SalehisadaghianiWeiPavel2019}, \cite{TatarenkoShiNedic_CDC2018}, \cite{Pavel2018} (the latter for generalized games), all exploiting a certain 
restricted
monotonicity property. Alternatively, the authors of \cite{TatarenkoNedic2019_unconstrained} developed a gradient-play scheme  by leveraging contractivity properties of doubly stochastic matrices. Nevertheless, in all these approaches theoretical guarantees are  provided only for step sizes that are typically very small, affecting the speed of convergence. 
Furthermore, all the methods cited are designed for a time-invariant, undirected network. To the best of our knowledge, switching communication topologies have only been addressed with diminishing step sizes.
For instance, the early work  \cite{Koshal_Nedic_Shanbag_2016}  considered aggregative games over time-varying jointly connected undirected graphs. This result was extended by the authors of \cite{BelgioiosoNedicGrammatico2020} to games with coupling constraints. In \cite{SalehisadaghianiPavel2017_nondoublystochastic}, an asynchronous gossip algorithm was presented to seek a \gls{NE} over directed graphs. The  
drawback is that  vanishing steps  typically result in slow convergence.

\emph{Contribution:} Motivated by the above, in this paper we present the first fixed-step \gls{NE} seeking algorithms for strongly monotone games  over time-varying communication networks. Our novel contributions are summarized as follows:
\begin{itemize}[topsep=-1pt,leftmargin=*]
	\item We propose a fully distributed projected gradient-play method, that is guaranteed to converge with linear rate when the network adjacency matrix is doubly stochastic. With respect to \cite{TatarenkoNedic2019_unconstrained},  we consider a time-varying communication network and we allow for constrained action sets. 
	Moreover, differently from the  state of the art, 
	we provide an upper bound on the step size that does not vanish as the number of agents increases  (§\ref{sec:distributedGNE});  
	\item We show via numerical simulations that, even in the case of fixed networks, our algorithm outperforms the existing pseudo-gradient based dynamics, when the step sizes are set to their theoretical upper bounds  (§\ref{sec:numerics});
	\item We prove that linear convergence to a \gls{NE} on time-varying weight-balanced directed graphs can be achieved via a forward-backward algorithm \cite[§12.7.2]{FacchineiPang2007}, which has  been studied in \cite{Pavel2018}, \cite{TatarenkoShiNedic_CDC2018}, but only for the special case of fixed undirected networks (§\ref{sec:balanced}).
\end{itemize}
\smallskip

\emph{Basic notation}: $\mathbb{N}$ is the set of natural numbers, including $0$. 
$\R$
is the set of
real numbers.
$\0_n$ ($\1_n$)  denotes the vector of dimension $n$ with all elements equal to $0$ ($1$); $I_n$  the identity matrix of dimension $n$; the subscripts might be omitted when there is no ambiguity. For a matrix $A \in \R^{m \times n}$, $A^\top$ denotes its transpose; $[A]_{i,j}$ is the element on row $i$ and column $j$, and the second subscript is omitted if $n=1$;
$\sigma_{\textnormal{min}}(A)=\sigma_1(A)\leq\dots\leq\sigma_n(A)=:\sigma_{\textnormal{max}}(A)=\|A\|$ denote  its singular values.  If $A\in\R^{n\times n}$, $\det(A)$ is its determinant;
$A \succ 0$  stands for symmetric positive definite matrix; if $A$ is symmetric,   $\uplambda_{\textnormal{min}}(A)=\uplambda_1(A)\leq\dots\leq\uplambda_n(A)=:\uplambda_{\textnormal{max}}(A)$ denote its eigenvalues.
$\otimes$ denotes the Kronecker product.
$\diag(A_1,\dots,A_N)$ denotes the block diagonal matrix with $A_1,\dots,A_N$ on its diagonal. Given $N$ vectors $x_1, \ldots, x_N$,  $x:=\col\left(x_1,\ldots,x_N\right) = [ x_1^\top \ldots  x_N^\top ]^\top$ and $x_{-i}=\col(x_1,\dots, x_{i-1},x_{i+1},\dots, x_N)$. $\| \cdot \|$ denotes the Euclidean vector norm.
For a differentiable function $g:\R^n \rightarrow \R$, $\nabla_{\!\!x} g(x)$ denotes its gradient. A mapping $\mc{A}:\R^m\rightarrow \R^n$ is $\ell$-Lipschitz continuous if, for any $x,y\in\R^m$, $\|\mc{A}(x)-\mc{A}(y)\|\leq \ell \|x-y\|$. $\proj _S:\R^n\rightarrow S $ denotes the Euclidean projection onto a closed convex set $S$. An operator $\mathcal{F} : \R^n \rightarrow \R^n$ is
($\mu$-strongly) monotone if, for any $x,y\in\R^n$, $(\mc{F}(x)-\mc{F}(y))^\top(x-y) \geq 0 \, (\geq \mu \|x-y\|^2 )$. The variational inequality VI$(\mc{F},S)$ is the problem of finding a vector $x^*\in {S}$ such that $\mc{F}(x^*)^\top(x-x^*)\geq 0$, for all $x\in S$.

\section{Mathematical setup}\label{sec:mathbackground}
We consider a set of  agents $ \mc I:=\{ 1,\ldots,N \}$, where each agent $i\in \mc{I}$ shall choose its action (i.e., decision variable) $x_i$  from its local  decision set $\textstyle \Omega_i \subseteq \R^{n_i}$. Let $x = \col( (x_i)_{i \in \mc I})  \in \Omega $ denote the stacked vector of all the agents' decisions, $\textstyle \Omega = \Omega_1\times\dots\times\Omega_N\subseteq \R^n$ the overall action space and $\textstyle n:=\sum_{i=1}^N n_i$. 
The goal of each agent $i \in \mc I$ is to minimize its objective function $J_i(x_i,x_{-i})$, which depends on both the local variable $x_i$ and the decision variables of the other agents $x_{-i}= \col( (x_j)_{j\in \mc I\backslash \{ i \} } )$.
The game is then  represented by the inter-dependent optimization problems:
\begin{equation} \label{eq:game}
	\forall i \in \mc{I}:
	\quad \underset{y_i \in \Omega_i}{\argmin}  \; J_i(y_i,x_{-i}).
\end{equation}
The technical problem we consider in this paper is the computation of a \gls{NE}, as defined next.

\begin{definition}
	A Nash equilibrium is a set of strategies $x^{*}=\operatorname{col}\left((x_{i}^{*})_{i \in \mathcal{N}}\right)\in \Omega$ such that, for all $ i \in \mathcal{I}$: 
	\[
	J_{i}\left(x_i^*, x_{-i}^{*}\right)\leq \inf \{J_{i}\left(y_{i}, x_{-i}^{*}\right) \mid y_i\in\Omega_i \}. \QEDopenhereeqn
	\]
\end{definition}
\smallskip

The following regularity assumptions are common for \gls{NE} problems, see, e.g., \cite[Ass.~1]{Pavel2018}, \cite[Ass.~1]{TatarenkoShiNedic_CDC2018}. 

\begin{standing}[Regularity and convexity]\label{Ass:Convexity}
	For each $i\in \mathcal{I}$, the set $\Omega_i$ is non-empty, closed and convex; $J_{i}$ is continuous  and the function $J_{i}\left(\cdot, x_{-i}\right)$ is convex and continuously differentiable for every $x_{-i}$.
	{\hfill \QEDopen} \end{standing}

Under Standing Assumption~\ref{Ass:Convexity}, a joint action $x^*$ is a \gls{NE} of the game in \eqref{eq:game} if and only 
if it solves the variational inequality VI$(F,\Omega)$ \cite[Prop.~1.4.2]{FacchineiPang2007}, or, equivalently, if and only if, for 
any $\alpha>0$  \cite[Prop.~1.5.8]{FacchineiPang2007},
\begin{equation} \label{eq:NEinclusion}
	x^*=\proj_{\Omega} (x^*-\alpha F(x^*)),
\end{equation}
where $F$ is the \emph{pseudo-gradient} mapping of the game: 
\begin{equation}
	\label{eq:pseudo-gradient}
	F(x):=\operatorname{col}\left( (\nabla _{\!\!  x_i} J_i(x_i,x_{-i}))_{i\in\mathcal{I}}\right).
\end{equation}
Next, we postulate a sufficient condition for  the existence of a unique \gls{NE},
namely the  strong monotonicity of the pseudo-gradient \cite[Th. 2.3.3]{FacchineiPang2007}. This assumption is always used for (G)\gls{NE} seeking under partial-decision information with fixed step  sizes, e.g., in \cite[Ass.~2]{TatarenkoShiNedic_CDC2018}, \cite[Ass.~3]{Pavel2018}.
It implies strong convexity of the functions $J_i(\cdot,x_{-i})$ for  every $x_{-i}$,  but not necessarily (strong) convexity of $J_i$ in the full argument.

\begin{standing}\label{Ass:StrMon}
	The pseudo-gradient mapping in \eqref{eq:pseudo-gradient}  is $\mu$-strongly monotone and $\ell_0$-Lipschitz continuous, for some $\mu$, $\ell_{0}>0$.
	\hfill \QEDopen
\end{standing}

In our setup, each agent $i$ can only access its own cost function $J_i$ and feasible set $\Omega_i$.
Moreover, agent $i$ does not have full knowledge of $x_{-i}$, and only relies on the information exchanged locally with neighbors over a time-varying directed communication network $\mathcal G_k(\mc{I},\mc{E}_k)$.  The ordered pair $(i,j) $ belongs to the set of edges, $\mc{E}_k$, if and only if agent $i$ can receive information from agent $j$ at time $k$. 
Let $W_{k}\in \R^{N\times N}$ denote the weighted adjacency matrix of $\mc{G}_k$, and $w^k_{i,j}:=[W_k]_{i,j}$, with $w^k_{i,j}>0$ if $(i,j)\in \mc{E}_k$, $w^k_{i,j}=0$ otherwise; $D_k=\diag((d_i^k)_{i\in\mc{I}})$ and $L_k=D_k-W_k$ the in-degree  and Laplacian matrices of $\mc{G}^k$, with $d_i^k=\textstyle \sum _{j=1}^N w_{i,j}^k$; 
$\mc{N}_i^k=\{j\mid (i,j)\in \mc{E}_k\}$ the set of in-neighbors of agent $i$. 

\begin{standing}\label{Ass:connectedgraph}
	For each $\k$, the graph $\mc{G}_k$ is strongly connected.  \hfill \QEDopen
\end{standing}

\begin{assumption}\label{Ass:doublestochasticity}
	\label{Ass:Graph2}
	For all $\k$, the following hold:
	\begin{itemize}[topsep=0pt]
		\item[(i)] \emph{Self-loops:} $w_{i,i}^k>0$ for all $i\in \mc{I}$;
		\item[(ii)] \emph{Double stochasticity:} $W_k\1_N=\1_N$, $\1_N^\top W_k=\1_N^\top$. 
		{\hfill \QEDopen}
	\end{itemize} 
\end{assumption}

\begin{remark}
	Assumption~\ref{Ass:Graph2}(i) is intended  just  to ease the notation. Instead,
	Assumption~\ref{Ass:Graph2}(ii) is stronger. It is typically used for networked problems on undirected symmetric graphs, e.g., in \cite[Ass.~6]{Koshal_Nedic_Shanbag_2016}, \cite[Ass.~3]{BelgioiosoNedicGrammatico2020}, \cite[Ass.~3]{TatarenkoNedic2019_unconstrained}, justified by the fact that it can be satisfied by assigning the following Metropolis weights to the communication:
	\[
	\tilde{w}_{i,j}^k=
	\left\{ 
	\begin{aligned}
	w_{i,j}^k/(\max\{d_i^k,d_j^k\}+1)  &&& \text{if} j\in\mc{N}_i \backslash \{i\}; \\
	0 &&& \text{if} j\notin\mc{N}_i ;\\
	1-\textstyle \sum_{j\in \mc{N}_i \backslash \{i\}} \tilde{w}_{i,j}^k &&&  \text{if} i=j.
	\end{aligned}  
	\right.
	\]
	In practice, to satisfy
	Assumption~\ref{Ass:Graph2}(ii) in case of symmetric communication, even under time-varying topology, it suffices for the agents to exchange their in-degree with their neighbors at every time step. Therefore, Standing Assumption~\ref{Ass:connectedgraph} and Assumption~\ref{Ass:Graph2} are easily fulfilled for undirected graphs connected at each step. For directed graphs, given any strongly connected topology, weights can be assigned such that the resulting adjacency matrix (with self-loops) is doubly stochastic, via an iterative distributed process \cite{GharesifardCortes_Doubly_2012}. However, this can be impractical if the network is time-varying. 
	\hfill \QEDopen
\end{remark}

Under Assumption~\ref{Ass:Graph2}, it holds that $\sigma_{N-1}(W_k)<1$, for all $k$, 
where $\sigma_{N-1}(W_k)$ denotes the second largest singular value 
of $W_k$. Moreover, for any $y\in \R^N$,
\begin{equation}\label{eq:Wcontractivity}
	\|W_k(y-\1_N\bar{y})\|\leq \sigma_{N-1}(W_k)\|y-\1_N \bar{y}\|,
\end{equation}
where $\bar{y}=\frac{1}{N}\1_N^\top y$ is the average of $y$. We will further assume
that $\sigma_{N-1}(W_k)$ is bounded away from 1; this automatically
holds if the networks $\mc{G}_k$ are chosen among a finite family.

\begin{assumption}\label{Ass:sigmabar}
	There exists $\bar{\sigma}\in (0,1)$ such that $\sigma_{N-1}(W_k)\leq {\bar{\sigma}}$, for all $k\in\N$.  \hfill \QEDopen
\end{assumption}

\section{Distributed  Nash equilibrium seeking}\label{sec:distributedGNE}
In this section, we present a pseudo-gradient algorithm 
to seek a \gls{NE} of the game \eqref{eq:game} in a fully distributed 
way. 
To cope with partial-decision information, each agent keeps an estimate of all other agents' actions. Let $\x_{i}=\operatorname{col}((\x_{i,j})_{j\in \mc{I}})\in \R^{Nn}$,  where $\x_{i,i}:=x_i$ and $\x_{i,j}$ is agent $i$'s estimate of agent $j$'s action, for all $j\neq i$; 
also, $\x_{j,-i}=\col((\x_{j,l})_{l\in\mc{I}\backslash \{ i \}})$. The agents aim at asymptotically reconstructing the true value of the opponents' actions, based on the data received from their neighbors. The procedure is summarized in Algorithm~\ref{algo:1}. Each agent updates its estimates according to consensus dynamics, then  its action via a gradient step. We remark that each agent computes the partial gradient of its cost in its local estimates $\bs x_i$, not on the actual joint action $x$.

To write the algorithm in compact form, let  $\x=\col((\x_i)_{i\in\mc{I}})$;  as in \cite[Eq.~13-14]{Pavel2018}, let,  for all $i \in \mc{I}$,
\begin{align}
	\mathcal{R}_{i}:=&\left[ \begin{array}{lll}{{\0}_{n_{i} \times n_{<i}}} & {I_{n_{i}}} & {\0_{n_{i} \times n_{>i}}}\end{array}\right] \in \R^{n_i\times n}, 
\end{align}
where
$n_{<i}:=\sum_{j=1}^{i-1}n_j$,
$n_{>i}:=\sum_{j=i+1}^{N}n_{j}$; 
let also $\mathcal{R}:=\operatorname{diag}\left((\mathcal{R}_{i})_{i \in \mathcal{I}}\right)\in\R^{n\times Nn}$.  In simple terms, $\mathcal R _i$ selects the $i$-th $n_i$ dimensional component from an $n$-dimensional vector. Thus, $\mathcal{R}_{i} \x_{i}=\x_{i,i}=x_i$,
and $x=\mathcal{R} \x$.
We define the  \textit{extended pseudo-gradient} mapping  $\bs{F}$ as
\begin{align}
	\label{eq:extended_pseudo-gradient}
	\bs{F}(\x):=\operatorname{col}\left((\nabla_{\!\!x_{i}} J_{i}\left(x_{i}, \x_{i,-i}\right))_{i \in \mathcal{I}}\right). 
\end{align}
Therefore, Algorithm~\ref{algo:1} reads in compact form as:
\begin{equation}\label{eq:algcompact}
	\x^{k+1}=\proj_{\bs{\Omega}}(\W_{\!\!k}\x^k-\alpha\mc{R}^\top\bs{F}(\W_{\!\!k}\x^k)),
\end{equation} 
where $\bs{\Omega}:=\{\x\in \R^{Nn} \mid \mc{R}\x\in\Omega \}$ and $\bs{W}_{\!\!k}:=W_k \otimes I_n$. 

\begin{lemma}[{\cite[Lemma 3]{Bianchi_ECC20_ctGNE}}]\label{lem:LipschitzExtPseudo}
	The  mapping $\bs{F}$ in \eqref{eq:extended_pseudo-gradient} is $\ell$-Lipschitz continuous, for some $\mu\leq\ell\leq \ell_0$. 
	{\hfill \QEDopen} \end{lemma}

\begin{theorem}\label{th:main1}
	Let Assumptions ~\ref{Ass:doublestochasticity}-\ref{Ass:sigmabar} hold and let 
	\begin{equation}\label{eq:M}
		M_{\alpha}  =  \left[
		\begin{matrix}
			1-\frac{2\alpha \mu}{N}+\frac{\alpha^2 \ell_0^2}{N} & & \left(\frac{\alpha(\ell+\ell_0)+\alpha^2\ell_0\ell}{\sqrt N}\right){\bar{\sigma}} \\
			\left(\frac{\alpha(\ell+\ell_0)+\alpha^2\ell_0\ell}{\sqrt N}\right){\bar{\sigma}} & & \left( 1+2\alpha\ell+\alpha^2\ell^2 \right) {\bar{\sigma}}^2
		\end{matrix}
		\right] .
	\end{equation}
	If the step size $\alpha>0$ is chosen such that 
	\begin{equation}\label{eq:C1}
		\rho_\alpha:=\uplambda_{\max}(M_{\alpha})=\|M_\alpha\|<1,
	\end{equation}
	then, for any initial condition, the sequence $(\bs{x}^k)_\k$ generated by Algorithm~\ref{algo:1} converges to $\x^*=\1_N\otimes x^*$, where $x^*$ is the 
	\gls{NE} of the game in \eqref{eq:game}, with linear rate: for all $\k$, 
	\[
	\| \x^k -\x^*\| \leq \left(\sqrt{\rho_\alpha}\right)^{\; k} \| \x^0-\x^*\|. \QEDopenhereeqn
	\]
\end{theorem}
%

\smallskip
\begin{lemma}\label{lem:bounds}
	The condition in \eqref{eq:C1} holds if $\alpha>0$ and 
	\begin{subequations}\label{eq:bounds}
				\newlength{\belowdisplayskipsave} \setlength{\belowdisplayskipsave}{\belowdisplayskip} \setlength{\belowdisplayskip}{0pt} 
		\begin{align}
			\alpha  &< \textstyle  \frac{{\bar{\sigma}}}{3\ell_0} \label{eq:condA} \\
			\alpha & < \textstyle \frac{2\mu }{\ell_0^2} \label{eq:condB} \\
			\nonumber
			0 & <2\mu(1-{\bar{\sigma}}^2)-\alpha ({\bar{\sigma}}^2(2\ell_0\ell
			+\ell^2+4\mu\ell+2\ell_0^2)-\ell_0^2) \\ & \quad
			-\alpha^2 (\ell_0\ell^2+\mu\ell^2+2\ell_0^2\ell)2{\bar{\sigma}}^2 -\alpha^32\ell_0^2\ell^2{\bar{\sigma}}^2. \label{eq:condC} 
		\end{align} 
				\setlength{\belowdisplayskip}{\belowdisplayskipsave} 
	\end{subequations}
	{\hfill $\square$}
\end{lemma}

\begin{proof}
	The condition in  \eqref{eq:condA} implies that   $M_{\alpha}\succ 0$ (by diagonal dominance and positivity of the diagonal elements{\color{black}, }as can be checked   by recalling that $\ell\leq \ell_0$, $\mu\leq \ell_0$, $N\geq 2$, $\sigmabar<1$).
	The  inequalities in \eqref{eq:condB}-\eqref{eq:condC} are the Sylvester's criterion for the matrix $I_2-M_\alpha$: they impose that  $[I_2-M_\alpha]_{1,1}>0$ \eqref{eq:condB} and $\det(I_2-M_\alpha)>0$ \eqref{eq:condC}, hence $I_2-M_{\alpha}\succ 0$.
	%
	%
	Altogether, this implies $\| M_\alpha \|<1$.
\end{proof}

\newlength{\textfloatsepsave} \setlength{\textfloatsepsave}{\textfloatsep} \setlength{\textfloatsep}{12pt}
\begin{algorithm}[t]\caption{Fully distributed \gls{NE} seeking} \label{algo:1}
	\vspace{0.3em}
	\textbf{Initialization}: for all $i\in \mc{I}$, set $x_i^0\in \Omega_i$, $\bs{x}_{i,-i}^0\in \R^{n-n_i}$. \vspace{0.2em}
	\\ 
	\textbf{Iterate until convergence:} for all $i\in\mc{I}$,
	\begin{itemize}[leftmargin=0.3em]
		\item[] 
		\emph{Distributed averaging: } $\hat{\x}_i^k = \sum_{j=1}^N w_{i,j}^k  \x_{j}^k $\vspace{0.3em}
		\item[]
		\emph{Local variables update:} 
		$
		\begin{aligned}[t]
		x_i^{k+1}&=\proj_{\Omega_i} (\xhb_{i,i}^k-\alpha \nabla_{\!\! x_i} J_i (\hat{\x}^k_i))\\
		\x_{i,-i}^{k+1}&=\hat{\x}^k_{i,-i}.
		\end{aligned}
		$
		\vspace{-0.2em}
	\end{itemize}
\end{algorithm}

\noindent 

\begin{remark}  
	The conditions in 	\eqref{eq:bounds} always hold for $\alpha $ small enough,
	since, in
	the monomial inequality \eqref{eq:condC},
	the constant term is $2\mu(1-{\bar{\sigma}}^2)>0$. While explicit solutions are known for cubic equations, we prefer the compact representation in \eqref{eq:condC}. The bounds in \eqref{eq:bounds} are not tight, and in practice  better bounds on the step size $\alpha$ are obtained by simply checking the Euclidean norm of the  $2\times 2$ matrix $M_\alpha$ in \eqref{eq:M}. 
	Instead, the  key observation is that the conditions in \eqref{eq:bounds} do not depend on the number of agents:  given the parameters ${\bar{\sigma}}$, $\mu$, $\ell_0$ and $\ell$,  a constant  $\alpha$ that ensures convergence can be chosen independently of  $N$. On the contrary, the rate $\sqrt{\rho_\alpha}$ does depend on $N$ and, in fact, it approaches $1$ as $N$ grows unbounded (analogously to the results in \cite{SalehisadaghianiWeiPavel2019}, \cite{TatarenkoShiNedic_CDC2018}, \cite{TatarenkoNedic2019_unconstrained}).
	\hfill  \QEDopen 
\end{remark}

\begin{remark}
	Compared to \cite[Alg. (7)]{TatarenkoNedic2019_unconstrained} (or \cite[Alg.~	1]{TatarenkoShiNedic_CDC2018}), in Algorithm~\ref{algo:1} the agents \emph{first} exchange information with their neighbors, and \emph{then}  evaluate their gradient term, resulting in better bounds on the step size $\alpha$. Moreover, differently from \cite[Th.~1]{TatarenkoNedic2019_unconstrained}, Theorem 1 provides a contractivity property for the iterates in \eqref{eq:algcompact} that holds at \emph{each} step. This has beneficial consequences in terms of robustness, see Remark~\ref{rem:ISS}. \hfill  \QEDopen
\end{remark}
\subsection{ Technical discussion }
In Algorithm~\ref{algo:1}, the partial gradients $\nabla_{\!\! x_i} J_i$ are evaluated 
on the local estimates $\x_{i,-i}$, not on the actual strategies
$x_{-i}$. Only if the estimates of all the agents coincide with the actual value, i.e., $\x =\1_N\otimes x$, we have that $\F(\x)=F(x)$. As a consequence, the mapping $\mc{R}^\top\F$ is not necessarily monotone, not even under strong monotonicity of the game mapping. Indeed, the loss of monotonicity is the main technical difficulty arising from the partial-decision information setup. Some works  \cite{GadjovPavel2018}, \cite{SalehisadaghianiWeiPavel2019}, \cite{Pavel2018}, \cite{Bianchi_ECC20_ctGNE} deal with this issue by leveraging a restricted strong monotonicity property,  which can be ensured, by opportunely choosing the parameter $\gamma$,  for the augmented mapping $\F_{\! \textnormal{a}} (\x):=\gamma \mc{R}^\top \F(\x) +\bs{L} \x$, where $\bs{L}=L\otimes I_n$ and $L$ is the Laplacian of a fixed undirected connected  network. 
Since the unique solution of the VI$(\Fa,\bs{\Omega})$ is $\x^*=\1_N\otimes x^*$, with $x^*$ the unique \gls{NE} of the game in \eqref{eq:game} \cite[Prop. ~1]{TatarenkoShiNedic_CDC2018}, one can design  \gls{NE} seeking algorithms 
via standard solution methods for variational inequalities (or the corresponding monotone inclusions, \cite{Pavel2018}).
For instance, in
\cite{TatarenkoShiNedic_CDC2018}, a forward-backward algorithm \cite[12.4.2]{FacchineiPang2007} is proposed 
to solve VI$(\Fa,\bs{\Omega})$, resulting in the  algorithm 
\begin{equation}\label{eq:algo:Nediccompact}
	\x^{k+1}=\proj _{\bs{\Omega}}\bigl( \x^k-\tau(\Fa(\x)) \bigr).
\end{equation}
We also recover this iteration when considering \cite[Alg.~1]{Pavel2018}
in the absence of coupling constraints. 
However, exploiting the monotonicity of $\Fa$ results in conservative 
upper bounds on the parameters $\tau$ and $\gamma$, and hence in slow convergence (see §\ref{sec:balanced}-\ref{sec:numerics}). 
More recently, the  authors of  \cite{TatarenkoNedic2019_unconstrained} studied the convergence of  \eqref{eq:algo:Nediccompact} based on contractivity of the iterates, in the case of a fixed undirected network with doubly stochastic adjacency matrix $W$, unconstrained action sets (i.e., $\Omega=\R^n$), and by fixing $\tau=1$, which results in 
the algorithm: 
\begin{equation}\label{eq:algo:unconstrainedcompact}
	\x^{k+1}=(W\otimes I_N)\x-\alpha\mc{R}^\top\bs{F}(\x^k).
\end{equation}
Nonetheless, the upper bound on $\alpha$ provided in \cite[Th.~1]{TatarenkoNedic2019_unconstrained} is decreasing 
to zero when the number of agents $N$  grows unbounded 
(in contrast with that in Theorem~\ref{th:main1}, see Lemma~\ref{lem:bounds}).

\section{Balanced directed graphs}\label{sec:balanced}
\noindent 
In this section, we relax Assumption~\ref{Ass:doublestochasticity} to the following.

\begin{assumption}\label{Ass:balanced}
	For all $\k$, the communication graph $\mc{G}_k$ is weight balanced: $(\1_N^\top W_k)^\top=W_k \1_N$. \hfill\QEDopen
\end{assumption}

For weight-balanced digraphs, in-degree and out-degree of each node coincide. Therefore, the matrix  $\tilde{L}_k:=(L_k+L_k^\top)/{2}=D_k-(W_k+W_k^\top)/{2}$  is itself the symmetric Laplacian of an undirected graph. Besides, such a graph is connected by Standing Assumption~\ref{Ass:connectedgraph}; hence $\tilde{L}_k$ has a simple eigenvalue in $0$, and the others are positive, i.e.,  $\uplambda_2(\tilde{L}_k)>0$.

\begin{assumption}\label{Ass:lambdaunderbar}
	There exist $\tilde{\sigma}$, $\bar{\uplambda} >0$ such that $\sigma_\textnormal{max}(L_k)\leq\tilde{\sigma}$ and $\uplambda_{2}(\tilde{L}_k)\geq {\bar{\uplambda}}$, for all $k\in\N$.  \hfill \QEDopen
\end{assumption}

\begin{remark}
	Assumptions~\ref{Ass:sigmabar} and \ref{Ass:lambdaunderbar} always hold if the networks switch among a finite family. 
	Yet, $\bar{\sigma}$, $\tilde{\sigma}$ and $\bar{\uplambda} $ are global parameters, that could be difficult to compute in a distributed way;  upper/lower bounds might be available for special classes of networks, e.g., unweighted graphs.
	\hfill \QEDopen
\end{remark}

\noindent To seek a \gls{NE} over switching balanced digraphs, we propose the iteration in Algorithm~\ref{algo:2}. In compact form, it reads as
\begin{equation}\label{eq:algo2:compact}
	\x^{k+1}=\proj _{\bs{\Omega}}\bigl( \x^k-\tau(\gamma \mc{R}^\top \F(\x^k) +\bs{L}_k\x^k) \bigr)
\end{equation}
where $\bs{L}_k=L_k\otimes I_n$. Clearly, \eqref{eq:algo2:compact} is the same scheme of \eqref{eq:algo:Nediccompact}, just adapted to take  the switching topology into account. In fact,  the proof of convergence of  Algorithm~\ref{algo:2} is based on a restricted strong monotonicity property of the operator
\begin{equation}\label{eq:Fa}
	\F_{\! \textnormal{a}}^k (\x):=\gamma \mc{R}^\top \F(\x) +\bs{L}_k \x,
\end{equation}
that still holds for balanced directed graphs, as we show next.

\begin{theorem}\label{th:main2}
	Let Assumptions~\ref{Ass:balanced}-\ref{Ass:lambdaunderbar} hold, and let 
	\begin{align}
		\label{eq:M1}
		\begin{aligned}
			{M}& :=\gamma\begin{bmatrix}{\frac{\mu}{N}} & \ {-\frac{\ell_0+\ell}{2\sqrt{N}}} \\ {-\frac{\ell_0+\ell}{2\sqrt{N}}} & \ { \textstyle \frac{ \bar{\uplambda}}{\gamma}-\theta}\end{bmatrix}, \\
			{\mubar}&:=\uplambda_{\textnormal{min}} ({M}),
		\end{aligned}
		\quad \begin{aligned}[c]
			\gammamax&:=\textstyle \frac{4\mu \bar{\uplambda}}{(\ell_0+\ell)^{2}+4\mu\theta},
			\\
			\bar{\ell}& :=\ell+\tilde{\sigma}
			\\
			\taumax &:= \textstyle{2\mubar}/{\bar{\ell}^2}, 
		\end{aligned}
		\\[0.2em]  
		\nonumber \rho_{\gamma,\tau} :=1-2\tau \mubar+\tau^2\bar{\ell}^2.~~~~~~~~~~~~
	\end{align}
	If  $\gamma\in (0,\gammamax)$, then $M\succ 0$ and, for any $\tau \in (0,\taumax)$, for any initial condition, the sequence $(\x^k)_\k$ generated by
	Algorithm~\ref{algo:2} converges to $\x^*=\1_N\otimes x^*$, where $x^*$ is 
	the unique \gls{NE} of the game in \eqref{eq:game}, with linear rate: for all $\k$, 
	\[
	\| \x^k -\x^*\| \leq \left(\sqrt{\rho_{\gamma,\tau}}\right)^{\; k} \| \x^0-\x^*\|. \QEDopenhereeqn
	\]
\end{theorem}

\medskip
\begin{remark}
	Differently from  the bound $\alphamax$ in \eqref{eq:M},  $\taumax$  in \eqref{eq:M1} vanishes as $N$ grows (fixed the other parameters), as   $\bar{\mu}$ decreases to  $0$ (by continuity of the eigenvalues).
	\hfill  \QEDopen 
\end{remark}

\begin{remark}\label{rem:ISS}
	Based on Theorems~\ref{th:main1}, \ref{th:main2}, it can be proven that the discrete-time systems \eqref{eq:algcompact},  \eqref{eq:algo2:compact} are \gls{ISS}  with respect to additive disturbances, with \gls{ISS}-Lyapunov function $\|\x-\x^*\|^2$. 
	By Lipschitz continuity of the updates, this  implies \gls{ISS} for noise both on the communication and in the evaluation of the  partial gradients.
	\hfill  \QEDopen 
\end{remark}

\begin{algorithm}[t]\caption{Fully distributed \gls{NE} seeking} \label{algo:2}
	\vspace{0.3em}
	\textbf{Initialization}: for all $i\in \mc{I}$, set $x_i^0\in \Omega_i$, $\bs{x}_{i,-i}^0\in \R^{n-n_i}$. \vspace{0.2em}
	\\ 
	\textbf{Iterate until convergence:} for all $i\in\mc{I}$,
	$$
	\begin{aligned}[t]
	\xhb_i^{k}&=\textstyle \sum^N _{j=1}w_{i,j}^k(\x_i^k-\x_{j}^k)
	\\
	x_i^{k+1}&=\proj_{\Omega_i} \bigl(x_i^k-\tau( \gamma \nabla_{\!\! x_i} J_i ({\x}_i^k)+\xhb_{i,i}^k) \bigr)\\
	\x_{i,-i}^{k+1}&=\hat{\x}^k_{i,-i}.
	\end{aligned}
	$$
	\vspace{-0.6em}
\end{algorithm} 

\section{Numerical example: A Nash-Cournot game}\label{sec:numerics}
We consider the  Nash-Cournot game in \cite[§6]{Pavel2018}. $N$ firms produce a commodity that is sold to $m$ markets. Each firm $i\in\mc{I}=\{1,\dots,N\}$ can only participate in $n_i\leq m$ of the markets; its action $x_i\in\R^{n_i}$ is the vector of quantities  of product to be sent to these $n_i$ markets, bounded by the local constraints $\0_{n_i}\leq x_i\leq X_i$. Let $A_i\in\R^{m\times n_i}$ be the matrix that specifies which markets firm $i$ participates in. Specifically, $[A_i]_{k,j}=1$ 
if $[x_i]_j$ is the amount of product sent to the $k$-th market by agent $i$,  $[A_i]_{k,j}=0$ otherwise, for all $k=1,\dots,m$, $j=1,\dots, n_i$. Let $A:=[A_1 \dots A_N]$; then $Ax=\textstyle \sum_{i=1}^N A_ix_i \in \R^m$ are the quantities of total product delivered to each market.  Firm $i$ aims at maximizing its profit, i.e.,  minimizing the cost function
$J_i(x_i,x_{-i})=c_i(x_i)-p(Ax)^\top A_ix_i.$
Here, $c_i(x_i)=x_i ^\top Q_i x_i+q_i^\top x_i$ is firm $i$'s production cost, with $Q_i\in \R^{n_i\times n_i}$, $Q_i\succ 0$, $q_i\in \R^{n_i}$. Instead, $p:\R^m\rightarrow \R^m$ associates to each market a price that depends on the amount of product delivered to that market. Specifically, the price for the market $k$, for $k=1,\dots,m$, is $ [p(Ax)]_k=\bar P_k$ -$\chi_k [Ax]_k$, where $\bar P_k$, $\chi_k>0$. 
We set $N=20$, $m=7$. The market structure is as in \cite[Fig.~1]{Pavel2018}, that defines which firms are allowed to participate in which  markets. Therefore, $x\in\R^n$, with $n=32$.  We select randomly with uniform distribution $r_k$ in $[1,2]$, $Q_i$ diagonal with diagonal elements in $[14,16]$, $q_i$ in $[1,2]$, $\bar{P}_k$ in $[10,20]$, $\chi_k$ in $[1,3]$, $X_i$  in $[5,10]$, for all $i\in \mc{I}$, $k=1,\dots,m$.
The resulting setup satisfies Standing Assumptions \ref{Ass:Convexity}-\ref{Ass:StrMon}  \cite[§VI]{Pavel2018}. 
The firms cannot access the production of all the competitors, but  can communicate with some neighbors on a network.

We first consider the case of a fixed, undirected graph, under Assumption~\ref{Ass:doublestochasticity}. Algorithm~\ref{algo:2} in this  case reduces to  \cite[Alg.~1]{TatarenkoShiNedic_CDC2018}.
We compare  Algorithms~\ref{algo:1}-\ref{algo:2} with the inexact ADMM in \cite{SalehisadaghianiWeiPavel2019} and the accelerated  gradient method in \cite{TatarenkoShiNedic_CDC2018}, for the step sizes that ensure convergence. Specifically, we set $\alpha$ as in Theorem~\ref{th:main1} for Algorithm~\ref{algo:1}. The convergence of all the other Algorithms is based on the monotonicity of $\Fa$ in \eqref{eq:Fa}; hence we set $\gamma$ as in Theorem~\ref{th:main2}.  Instead of using the conservative bounds in \eqref{eq:M1} for the parameters, $\bar{\mu}$ and $\bar{\ell}$, we obtain a better result by computing the values numerically.  $\Fa$ is (non-restricted) strongly monotone for our parameters, hence also the convergence result for \cite[Alg.~2 ]{TatarenkoShiNedic_CDC2018} holds. Figure~\ref{fig:1} shows that Algorithm~\ref{algo:1} outperforms all the other methods (we also note that the accelerated gradient in \cite[Alg.~2]{TatarenkoShiNedic_CDC2018} requires two projections and two communications per iterations). 
As a numerical example, we also compare Algorithm~\ref{algo:1} with the scheme in \eqref{eq:algo:unconstrainedcompact} by removing the local constraints, in Figure~\ref{fig:2}.

\setlength{\textfloatsep}{\textfloatsepsave-1em}
\begin{figure}[t]
	\centering
	\includegraphics[width=1\columnwidth]{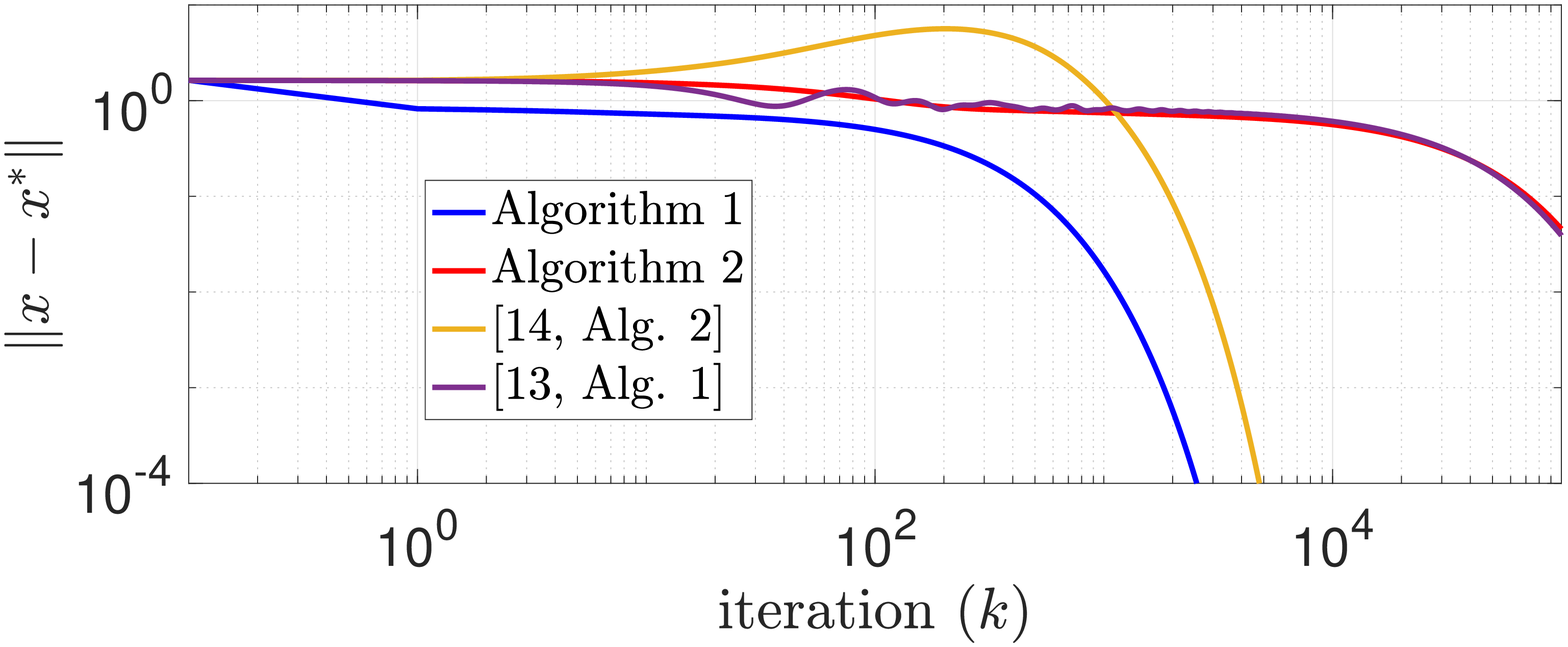}
	\caption{Distance from the \gls{NE} for different pseudo-gradient \gls{NE} seeking methods, with  step sizes that guarantee convergence. }\label{fig:1}
\end{figure}
\begin{figure}[t]
	\centering
	\includegraphics[width=1\columnwidth]{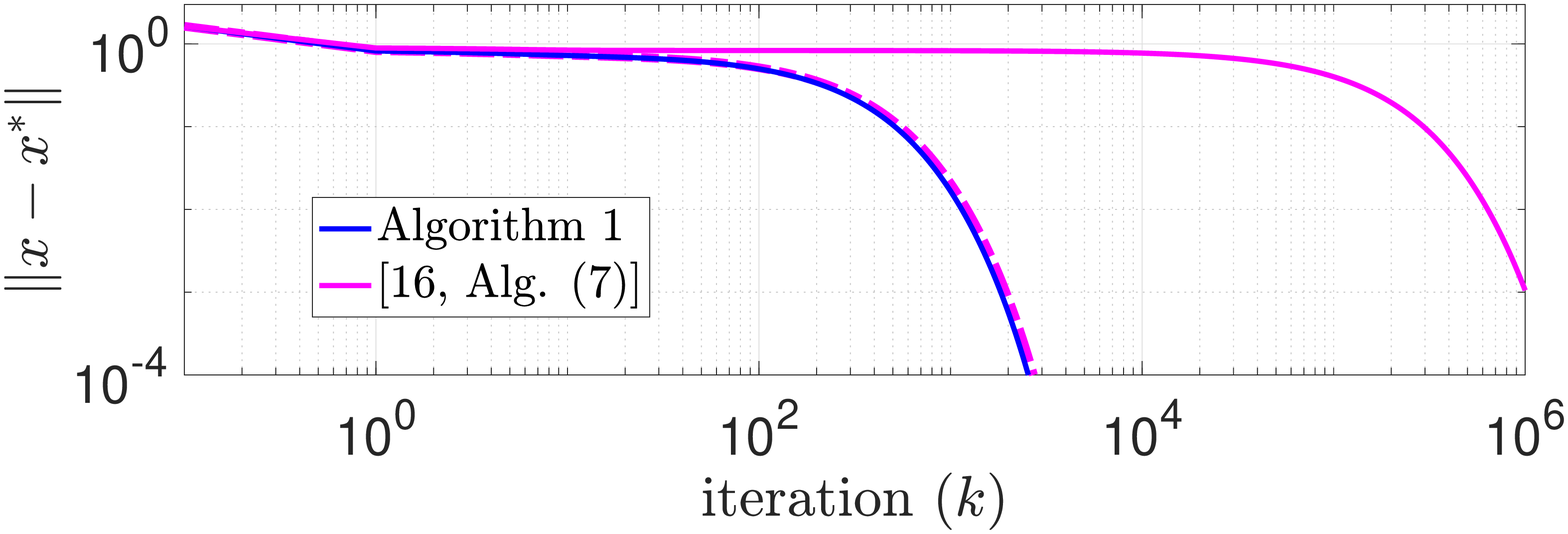}
	\caption{Distance from the NE for Algorithm~\ref{algo:1}, with step size $\alpha=2*10^{-3}$ (upper bound in Theorem~\ref{th:main1}), and the method in  \cite[Alg.~1]{TatarenkoNedic2019_unconstrained}\label{fig:2}, with step size $\alpha=4*10^{-6}$ (upper bound in \cite[Th.~1]{TatarenkoNedic2019_unconstrained}). Algorithm~\ref{algo:1} converges much faster, thanks to the  larger step size. The scheme in \cite[Alg.~1]{TatarenkoNedic2019_unconstrained} still converges if we set $\alpha=2*10^{-3}$ (dashed line, not supported theoretically).}
\end{figure}
\begin{figure}[t]
	\centering
	\includegraphics[width=1\columnwidth]{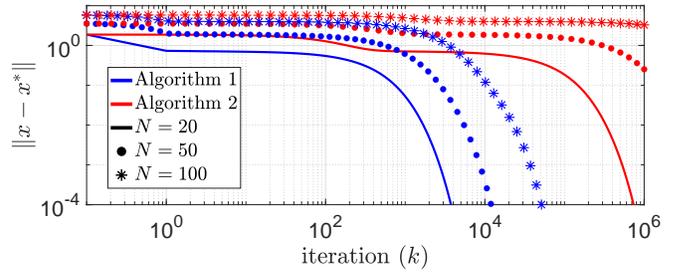}
	\caption{Comparison of Algorithms~1 and \ref{algo:2}, on a time-varying graph, for $20$, $50$ or $100$ agents, with the step sizes set to their theoretical upper bounds.} \label{fig:3}
\end{figure}
\begin{figure}[t]
	\centering
	\includegraphics[width=1\columnwidth]{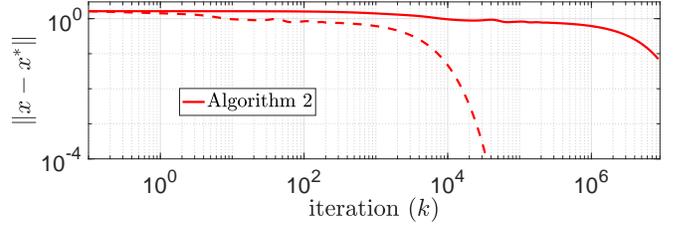}
	\caption{Distance from the \gls{NE} for Algorithm~\ref{algo:2}, on a time-varying digraph. Since the networks are sparse, Theorem~\ref{th:main2} ensures convergence only for small step sizes ($\gamma=5*10^{-4}$, $\tau=3*10^{-4}$), and convergence is slow (solid line). However, the bounds are conservative: the iteration still converges with $\tau$ $1000$ times larger than the theoretical value (dashed line). }\label{fig:4}
\end{figure}
For the case of doubly stochastic time-varying networks, we randomly generate $5$ connected graphs and for each iteration we pick one with uniform distribution. In Figure~\ref{fig:3}, we compare the performance of Algorithms~\ref{algo:1}-\ref{algo:2}, for step sizes  set to their upper bounds as in Theorems~\ref{th:main1}-\ref{th:main2}. Since the theoretical convergence rate in Theorems~\ref{th:main1}-\ref{th:main2} worsens as the number of agents grows, to show how the performance is affected in practice, we repeat the experiment for different values of $N$ and random market structures (Figure~\ref{fig:3}).
\newline \indent
Finally, in  Figure~\ref{fig:4},
we test Algorithm~\ref{algo:2} with communication topology chosen at each step with uniform distribution between two unweighted balanced directed graphs: the directed ring, where each agent $i$ can send information to the agent $i+1$ (with the convention $N+1\equiv1$), and a  graph where agent $i$
is also allowed to transmit to agent $i+2$, for all $i\in\mc{I}$.

\section{Conclusion}\label{sec:conclusion}
Nash equilibrium problems on time-varying graphs can be solved  with linear rate via  fixed-step pseudo-gradient algorithms, 
if the network is connected at every iteration and the game mapping is  Lipschitz continuous and strongly monotone. Our algorithm proved much faster than the existing gradient-based methods, when the step sizes satisfy their theoretical upper bounds.
The extension 
to games with coupling constraints is left as future research. It would be also valuable to relax our uniform connectedness assumption, i.e., allowing for jointly strongly connected directed graphs.
%

\appendix
\subsection{Proof of Theorem~\ref{th:main1}}\label{app:th:main1}
\noindent We define the estimate consensus subspace $\bs{E}:=\{ \bs{y}\in\R^{Nn} \mid \y=\1_N\otimes y, \ y\in\R^n \}$ 
and its orthogonal complement $\bs{E}_{\perp}=\{ \bs{y}\in\R^{Nn}\mid (\1_N\otimes I_n)^\top \y=\0_n \}$. Thus, any vector $\x\in\R^{Nn}$ can be written as $\x=\x_{\scriptscriptstyle \parallel}+\x_{\! \scriptscriptstyle \perp}$, where $\x_{\scriptscriptstyle \parallel}=\proj_{\bs E}(\x)=\frac{1}{N}(\1_N\1_N^\top\otimes I_n) \x$, $\x_{\! \scriptscriptstyle \perp}=\proj_{\bs E_{\perp}}(\x)$, and $\x_{\scriptscriptstyle \parallel}^\top\x_{\! \scriptscriptstyle \perp}=0$. Also, we  use the shorthand notation $\F\x$ and $Fx$ in place of $\F(\x)$ and $F(x)$. 
We recast the iteration in \eqref{eq:algcompact} as
\begin{equation}\label{eq:algoperators} 
	\x^{k+1}=\proj_{\bs{\Omega}} (\xhb^k-\alpha\mc{R}^\top \bs{F}\xhb^k), \  \xhb^k=\W_{\!\!k} \x^k.
\end{equation}
Let $x^*$ be the unique \gls{NE} of the game in \eqref{eq:game}, and $\x^*=\1_N\otimes x^*$. We recall that $x^*=\proj_{{\Omega}}(x^*-\alpha Fx^*)$ by \eqref{eq:NEinclusion}, and then $\x^*=\proj_{\bs{\Omega}}(\x^*-\alpha\mc{R}^\top\F\x^*)$. Moreover, $\bs{W}_{\!\! k}\x^*=(W_k\otimes I_n)(\1_N\otimes x^*)=\1_N\otimes x^*=\x^*$;  hence $\x^*$ is a
fixed point for \eqref{eq:algoperators}. 
Let $\x^k=\x\in\R^{Nn}$ and $\xhb=\bs{W}_{\!\!k}\x=\xhb_{\scriptscriptstyle \parallel }+\xhb_{\! \scriptscriptstyle \perp}=\1_N\otimes \hat{x}_{\scriptscriptstyle \parallel}+\xhb_{\! \scriptscriptstyle \perp}\in\R^{Nn}$. Thus,  it holds that 
%
%
\allowdisplaybreaks
\begin{align} 
	\nonumber 
	& \quad ~ \|\x^{k+1}-\x^*\|^2
	\\ \nonumber
	& =  \|  \proj_{\bs{\Omega}}(\xhb-\alpha\mc{R}^\top \bs{F}\xhb)-  \proj_{\bs{\Omega}}(\x^*-\alpha\mc{R}^\top \bs{F}\x^*)\|^2
	\\ \nonumber
	& \leq  \|  (\xhb-\alpha\mc{R}^\top \bs{F}\xhb)-  (\x^*-\alpha\mc{R}^\top \bs{F}\x^*)\|^2 
	\\\nonumber
	& = \begin{multlined}[t]
		\|\xhb_{\scriptscriptstyle \parallel }+\xhb_{\! \scriptscriptstyle \perp}-\x^*+\alpha\mc{R}^\top (-\bs{F}\xhb+ \bs{F}\x^* + \bs{F}\xhb_{\scriptscriptstyle \parallel} 
		- \bs{F}\xhb_{\scriptscriptstyle \parallel}) \|^2
	\end{multlined}
	\\[0.02pt] \nonumber
	\\
	&\begin{aligned}[c]
		&= \|\xhb_{\scriptscriptstyle \parallel}-\x^*\|^2+\|\xhb_{\! \scriptscriptstyle \perp} \|^2 \\ 
		&\quad \ +\alpha^2  \|\mc{R}^\top(\F\xhb-\F\xhb_{\scriptscriptstyle \parallel}+\F\xhb_{\scriptscriptstyle \parallel}-\F \x^*)\|^2 \\ 
		& \quad \ -2\alpha(\xhb_{\scriptscriptstyle \parallel}-\x^*)^\top \mc{R}^\top (\F\xhb-\F \xhb_{\scriptscriptstyle \parallel}) \\
		& \quad \ -2\alpha(\xhb_{\scriptscriptstyle \parallel}-\x^*)^\top\mc{R}^\top(\F\xhb_{\scriptscriptstyle \parallel}-\F\x^*)
		\\ & \quad \
		-2\alpha \xhb_{ \scriptscriptstyle \perp}^{ \top}\mc{R}^{\! \top}(\F \xhb-\F \xhb_{\scriptscriptstyle \parallel}) 
		\\& \quad \
		-2\alpha\xhb_{\scriptscriptstyle \perp}^{ \top}\mc{R}^{\top}(\F\xhb_{\scriptscriptstyle \parallel}-\F\x^*) 
	\end{aligned}  \label{eq:addends}
	\\\nonumber
	& \leq \begin{aligned}[t]
		& \|\xhb_{\scriptscriptstyle \parallel}-\x^*\|^2+\|\xhb_{\! \scriptscriptstyle \perp}\|^2+\alpha^2(\ell^2\|\xhb_{\! \scriptscriptstyle \perp}\|^2+\textstyle \frac{\ell_0^2}{N} \|\xhb_{\scriptscriptstyle \parallel}-\x^*\|^2
		\\ &
		+\textstyle \frac{2\ell_0\ell }{\sqrt{N}}\|\xhb_{\scriptscriptstyle \parallel}-\x^*\|\|\xhb_{\! \scriptscriptstyle \perp}\|)+\textstyle \frac{2\alpha\ell }{\sqrt{N}}\|\xhb_{\scriptscriptstyle \parallel}-\x^*\|\|\xhb_{\! \scriptscriptstyle \perp}\|
		\\&
		\textstyle -\frac{2\alpha\mu}{N}\|\xhb_{\scriptscriptstyle \parallel}-\x^*\|^2+2\alpha\ell\|\xhb_{\! \scriptscriptstyle \perp}\|^2 +\frac{2\alpha\ell_0 }{\sqrt{N}}\|\xhb_{\! \scriptscriptstyle \perp}\|\|\xhb_{\scriptscriptstyle \parallel}-\x^*\|,
	\end{aligned}
\end{align}
where the first inequality follows by nonexpansiveness of the projection (\cite[Prop.~4.16]{Bauschke2017}), 
and to bound the addends in \eqref{eq:addends} we used, in the order:
\begin{itemize}[nosep,noitemsep,leftmargin=*,topsep=0pt]
	\item 3\textsuperscript{rd} term: $\|\mc{R}\|=1$, Lipschitz continuity of $\bs{F}$, and $\|\F\xhb_{\scriptscriptstyle \parallel}-\F \x^*\|=\|F\hat{x}_{\scriptscriptstyle \parallel}-F x^*\|\leq \ell_0 \|\hat{x}_{\scriptscriptstyle \parallel}-x^*\|=\frac{\ell_0}{\sqrt{N}}\|\xhb_{\scriptscriptstyle \parallel}- \x^*\|$;
	\item 4\textsuperscript{th} term: $\|\mc{R}(\1\otimes (\hat{x}_{\scriptscriptstyle \parallel}-x^*))\|=\| \hat{x}_{\scriptscriptstyle \parallel}-x^*\|=\textstyle \frac{1}{\sqrt{N}}\|\xhb_{\scriptscriptstyle \parallel}-\x^*\|$;
	\item 5\textsuperscript{th} term: $(\xhb_{\scriptscriptstyle \parallel}-\x^*)^\top\mc{R}^\top(\F\xhb_{\scriptscriptstyle \parallel}-\F\x^* )=(\hat{x}_{\scriptscriptstyle \parallel}-x^*)^\top(F \hat{x}_{\scriptscriptstyle \parallel}-Fx^*)\geq \mu \|\hat{x}_{\scriptscriptstyle \parallel}-x^*\|^2=\frac{1}{N}\|\xhb_{\scriptscriptstyle \parallel}-\x^*\|^2$;
	\item 6\textsuperscript{th} term: Lipschitz continuity of $\F$;
	\item 7\textsuperscript{th} term: $\|\F\xhb_{\scriptscriptstyle \parallel}-\F\x^*\|\leq\textstyle \frac{\ell_0}{\sqrt{N}}\|\xhb_{\scriptscriptstyle \parallel}-\x^*\|$ as above.
\end{itemize}
Besides, for every $\x=\x_{\scriptscriptstyle \parallel}+\x_{\! \scriptscriptstyle \perp}\in \R^{Nn}$ and for all $\k$, it holds that $ \xhb= \W_{\!\!k}\x=\x_{\scriptscriptstyle \parallel}+\W_{\!\!k} \x_{\! \scriptscriptstyle \perp}$, where $\W_{\!\!k} \x_{\! \scriptscriptstyle \perp} \in \bs{E}_{\! \scriptscriptstyle \perp}$, by doubly stochasticity of $W_{k}$, and $\| \xhb_{\! \scriptscriptstyle \perp}\|=\| \W_{\!\!k} \x_{\! \scriptscriptstyle \perp} \|\leq {\bar{\sigma}} \|\x_{\! \scriptscriptstyle \perp}\|$ by \eqref{eq:Wcontractivity} and properties of the Kronecker product.
Therefore we can finally write, for all $\k$, for all $\x^k\in\R^{Nn}$,
\[
\begin{aligned}[b]
\ \|\x^{k+1}-\x^*\|^2 
& \leq \left[\begin{matrix}
\|\x^k_{\scriptscriptstyle \parallel} -\x^*\| \\\nonumber \| \x^k_{\! \scriptscriptstyle \perp} \|
\end{matrix}\right]^\top 
M_\alpha 
\left[\begin{matrix}
\|\x^k_{\scriptscriptstyle \parallel} -\x^*\| \\\nonumber \|\x^k_{\! \scriptscriptstyle \perp}\|
\end{matrix}\right] 
\\
&  \leq\uplambda_{\max}(M_\alpha)(\|\texttt{}\x^k_{\scriptscriptstyle \parallel} -\x^*\|^2+\|\x^k_{\! \scriptscriptstyle \perp}\|^2) 
\\
&= \uplambda_{\max}(M_\alpha)\|\x^k-\x^*\|^2.
\end{aligned}  \QEDhereeqn
\]
\subsection{Proof of Theorem~\ref{th:main2}}\label{app:th:main2}

\noindent Let $x^*$ be the unique \gls{NE} of the game in \eqref{eq:game}, and $\x^*=\1_N\otimes x^*$. We recall that the null space $\Null(\bs{L}_k)=\bs{E} =\{ \bs{y}\in\R^{Nn} \mid \y=\1_N\otimes y, \ y\in\R^n \}$ by Standing Assumption~\ref{Ass:connectedgraph}.  Therefore, $\bs{L}_k \x^*=\0_N$ and $\x^*$ is a fixed point of the iteration in \eqref{eq:algo2:compact}  by \eqref{eq:NEinclusion}. With $\Fa^k$ as in \eqref{eq:Fa}, for all $\k$, for any $\x\in\R^{Nn}$, it holds that
$(\x-\x^*)^\top (\Fa^k\x-\Fa^k\x^*)
=(\x-\x^*)^\top \gamma \mc{R}^\top(\F\x-\F\x^*)+(\x-\x^*)^\top \bs{L}_k(\x-\x^*)
=(\x-\x^*)^\top \gamma \mc{R}^\top(\F\x-\F\x^*)+(\x-\x^*)^\top \tilde{\bs{L}}_k(\x-\x^*),
$
where $\tilde{\bs{L}}_k=({\bs{L}}_k+{\bs{L}}_k^\top)/2=({{L}}_k+{{L}}_k^\top)\otimes I_n /2=\tilde{L}_k\otimes I_n$, and $\tilde{L}_k$ is the Laplacian of a connected graph (see §\ref{sec:balanced}) and $\uplambda_2(\tilde{L}_k)>\bar{\uplambda}$ by Assumption~\ref{Ass:lambdaunderbar}. Therefore we can apply \cite[Lemma 3]{Pavel2018} to conclude that $
(\x-\x^*)^\top (\Fa^k\x-\Fa^k\x^*)\geq  \mubar\|\x-\x^*\|^2,
$
with $\mubar>0$ as in \eqref{eq:M1}. Also,  $\Fa^k$  is Lipschitz continuous with constant $\bar{\ell}=\ell+\tilde{\sigma}$, $\tilde{\sigma}$ as in Assumption~\ref{Ass:lambdaunderbar}. Therefore we have
\[ 
\begin{aligned}
& \quad~ \|\x^{k+1}-\x^*\|^2 \\
& = \|\proj _{\bs{\Omega}}(\x^k-\tau\Fa^k(\x^k) )- \proj _{\bs{\Omega}}\bigl(\x^*-\tau \Fa^k\x^*) \|^2\ \\
& \leq \|(\x^k-\tau\Fa^k\x^k) -(\x^*-\tau \Fa^k\x^*)\|^2 \\
& = \begin{multlined}[t]
\|\x^k-\x^*\|^2-2\tau(\x^k-\x^*)^\top(\Fa^k\x^k-\Fa^k\x^*)\\ +\tau^2\|\Fa^k\x^k-\Fa^k\x^*\|^2
\end{multlined} \\
& \leq (1-2\tau \mubar+\tau^2(\ell+\tilde{\sigma})^2 ) \| \x^k-\x^*\|^2= \rho_{\gamma,\tau}\| \x^k-\x^*\|^2,
\end{aligned}
\]
where in the first inequality we used  \cite[Prop.~4.16]{Bauschke2017}, 
and $\rho_{\gamma,\tau}\in (0,1)$ if $\tau$ is chosen as in Theorem~\ref{th:main2}. \hfill $\QEDclosed$
\vspace{-0.1em}
\bibliographystyle{IEEEtran}
\bibliography{library}
\end{document}